\documentclass[11pt]{article}
\usepackage{fullpage}
\usepackage{amsthm}
\usepackage{mathptmx}
\newtheorem{fact}{Fact}
\newtheorem{corollary}{Corollary}

\begin{document}

\title{On the distance between non-isomorphic groups}
\author{
G\'abor Ivanyos\thanks{
Computer and Automation Research Institute of
the Hungarian Academy of Sciences,
Kende u.~13-17, H-1111 Budapest, Hungary.
E-mail: {\tt Gabor.Ivanyos@sztaki.hu}
}
\and 
Fran{\c c}ois Le Gall\thanks{
Department of Computer Science, The University of Tokyo,
7-3-1 Hongo, Bunkyo-ku, Tokyo 113-8656, Japan.
E-mail: {\tt legall@is.s.u-tokyo.ac.jp}
}
\and 
  Yuichi Yoshida\thanks{
  School of Informatics, Kyoto University,
  and Preferred Infrastructure, Inc.,
  Yoshida-Honmachi, Kyoto 606-8501, Japan.
E-mail: {\tt yyoshida@kuis.kyoto-u.ac.jp}
}
}

\date{}
\maketitle

\begin{abstract}
A result of Ben-Or, Coppersmith, Luby and Rubinfeld on testing whether a map between
two groups is close to a homomorphism implies
a tight lower bound on the distance between the multiplication tables
of two non-isomorphic groups.
\end{abstract}

In~\cite{Drapal92} Dr\'apal showed that if
$\circ$ and $\ast$ are two binary operations
on the finite set $G$ such that $(G,\circ)$
and $(G,\ast)$ are non-isomorphic groups
then the Hamming distance between the two 
multiplication tables is greater than $\frac{1}{9}|G|^2$.
In~\cite{Drapal03} 
infinite families of non-isomorphic pairs of 3-groups
with distance exactly $\frac{2}{9}|G|^2$ 
are given.

In this note we show that
$\frac{2}{9}|G|^2$ is a lower bound
for the distance of arbitrary non-isomorphic
group structures. The proof is a simple application
of the following result from \cite{Ben-Or+08}.

\begin{fact}\label{fact:Ben-Or}
  Let $(G,\circ)$ and $(K,\ast)$ be two groups 
and $f\colon G\to K$ be a map such that 
$$\frac{\#\left\{(x,y)\in G\times G:f(x\circ y)=f(x)\ast
f(y)\right\}}{|G|^2}>\frac{7}{9}.$$
  Then there exists  a group homomorphism $h\colon G\to K$ such that
   $\frac{\#\left\{x\in G:f(x)= h(x)\right\}}{|G|}\ge\frac{5}{9}\:$.
\end{fact}

Fact~\ref{fact:Ben-Or} is a weak version of
Theorem~1 in \cite{Ben-Or+08}. Here is a brief sketch of its proof.
For every $x\in G$, $h(x)$ is defined as
the value taken most frequently  by the expression 
$f(x\circ y)\ast f(y)^{-1}$ where $y$ runs over $G$. Then the 
first step is showing that for every $x\in G$,
$\#\{y\in G:f(x\circ y)\ast f(y)^{-1}=h(x)\}>\frac{2}{3}|G|$. The
homomorphic property of $h$ and equality of $h(x)$ with $f(x)$ 
for $\frac{5}{9}$ of the possible elements $x$
follow from this
claim easily.

We apply Fact~\ref{fact:Ben-Or} to obtain a result
on the distance of multiplication tables of
groups of not necessarily equal size. It will be
convenient to state it in terms of a quantity
complementary to the distance. Let $(G,\circ)$ and $(K,\ast)$ 
be finite groups. We define
the overlap between $(G,\circ)$ and $(K,\ast)$ as
$$\max_{\gamma:G\hookrightarrow S,\kappa:K\hookrightarrow S}
\#\left\{(x,y)\in G\times G:
\exists(x',y')\in K\times K
\mbox{~s.t.~}
\begin{array}{c}
\gamma(x)=\kappa(x'),\\
\gamma(y)=\kappa(y'),\\
\gamma(x\circ y)=\kappa(x'\ast y')
\end{array}
\right\},
$$
where $S$ is any set with $|S|\geq \max(|G|, |K|)$.
\begin{corollary}\label{theorem}
If $|G|\leq |K|$ and $(G,\circ)$ is not isomorphic to a subgroup of
 $(K,\ast)$ then the overlap between $(G,\circ)$ and $(K,\ast)$ 
is at most $\frac{7}{9}|G|^2$.
\end{corollary}

\begin{proof}
Assume that the
overlap is larger than $\frac{7}{9}|G|^2$. Then
there exist injections
$\gamma:G\hookrightarrow S,\kappa:K\hookrightarrow S$
such that the set
$$Z=\left\{(x,y)\in G\times G:
\exists(x',y')\in K\times K
\mbox{~s.t.~}
\begin{array}{c}
\gamma(x)=\kappa(x'),\\
\gamma(y)=\kappa(y'),\\
\gamma(x\circ y)=\kappa(x'\ast y')
\end{array}
\right\}$$
has cardinality larger than $\frac{7}{9}|G|^2$.
Put $$G_0=\{x\in G|\exists x'\in K \mbox{~such that~}
\gamma(x)=\kappa(x')\}.$$ Then $\kappa^{-1}\circ \gamma$
embeds $G_0$ into $K$ and it can be extended to an
injection $\phi: G\hookrightarrow K$. For $(x,y)\in Z$
we have $$\phi(x\circ y)=\kappa^{-1}(\gamma(x\circ y))=
\kappa^{-1}(\gamma(x))\ast \kappa^{-1}(\gamma(y))
=\phi(x)\ast \phi(y),$$ therefore, by Fact~\ref{fact:Ben-Or}, 
there exists a homomorphism
$\psi:G\rightarrow K$ such that
$$\#\{x\in G:\psi(x)\neq \phi(x)\}<\frac{4}{9}|G|.$$
This, together with the injectivity of $\phi$
implies the $\psi$ is injective as well and its image is
a subgroup of $(K,\ast)$ isomorphic to $(G,\circ)$.
\end{proof}


\begin{thebibliography}{1}
\bibitem{Ben-Or+08}
M.~Ben-Or, D.~Coppersmith, M.~Luby, R.~Rubinfeld, Non-abelian homomorphism
  testing, and distributions close to their self-convolutions. 
{Random Structures and Algorithms},
  32 (2008), 49--70.
\bibitem{Drapal92}
A.~Dr{\' a}pal,  How far apart can the group multiplication tables be?,
{European Journal of Combinatorics} 13 (1992), 335--343.
\bibitem{Drapal03}
A.~Dr{\' a}pal, On distances of 2-groups and 3-groups,
in: {C. M. Campbell, E. F. Robertson, G. C. Smith (Eds.),
Groups St Andrews 2001 in Oxford: Volume 1 (LMS
Lecture Notes Series No.~304)}, Cambridge University Press, Cambridge, 
2003, pp.~143--149.  
 \end{thebibliography}
 \end{document}